\documentclass[reqno,11pt]{amsart}
\usepackage{amsfonts}
\usepackage{amsmath,amsthm} 
\usepackage{hyperref} 
\usepackage{latexsym}
\usepackage{array}
\usepackage{amssymb}
\usepackage{enumerate}
\usepackage{bbold}
\usepackage[francais,english]{babel}
\usepackage{color}
\usepackage{stmaryrd}
\usepackage[normalem]{ulem}

\theoremstyle{plain}  \newtheorem{theorem}{Theorem}[section]
\newtheorem{lemma}[theorem]{Lemma}
\newtheorem{proposition}[theorem]{Proposition}
\newtheorem{hypothesis}[theorem]{Hypothesis}
\newtheorem{corollary}[theorem]{Corollary} 
\newtheorem{definition}[theorem]{Definition} \theoremstyle{remark}
\newtheorem{remark}[theorem]{Remark}

\newcommand{\Rc}{\mathcal{R}}
\newcommand{\diag}{\operatorname{diag}}

\newcommand{\be}{\begin{equation}}
\newcommand{\ee}{\end{equation}}

\newcommand{\ab}{\boldsymbol{a}}
\newcommand{\bb}{\boldsymbol{b}}

\newcommand{\dd}{\mathrm{d}}

\newcommand{\Hc}{\mathcal{H}}

\newcommand{\M}{\mathcal{M}}

\newcommand{\jb}{{\boldsymbol{j}}}

\newcommand{\kb}{{\boldsymbol{k}}}

\newcommand{\ellb}{{\boldsymbol{\ell}}}

\renewcommand{\Mc}{\mathcal{M}}

\newcommand{\R}{\mathbb{R}}

\newcommand{\C}{\mathbb{C}}

\newcommand{\T}{\mathbb{T}}
\newcommand{\Z}{\mathbb{Z}}

\newcommand{\Sc}{\mathcal{S}}

\newcommand{\Uc}{\mathcal{U}}

\newcommand{\U}{\mathbb{U}}

\newcommand{\eps}{\varepsilon}
\newcommand{\om}{\omega}

\newcommand{\Norm}[2]{\|#1\|\left.\vphantom{T_{j_0}^0}\!\!\right._{#2}}

\author{Joackim Bernier }
\address{Univ Rennes, INRIA, CNRS, IRMAR - UMR 6625, F-35000 Rennes, France} 
\email{Joackim.Bernier@ens-rennes.fr}

\author{Erwan Faou}
\address{Univ Rennes, INRIA, CNRS, IRMAR - UMR 6625, F-35000 Rennes, France} 
\email{Erwan.Faou@inria.fr}

 \author{ Beno\^it Gr\'ebert}
\address{Laboratoire de Math\'ematiques Jean Leray, Universit\'e de Nantes, UMR CNRS 6629\\
2, rue de la Houssini\`ere \\
44322 Nantes Cedex 03, France}
\email{benoit.grebert@univ-nantes.fr}

\title[Long time behavior of NLW]{Long time behavior of the  solutions of NLW  on the d-dimensional torus}

\begin{document}

\begin{abstract}
We consider the non linear wave equation (NLW) on the $d$-dimensional torus  $\T^d$ with a smooth nonlinearity of order at least two at the origin. We prove  that, for almost any mass, small and smooth solutions of high Sobolev indices are stable up to arbitrary long times with respect to the size of the initial data.
To prove this result we use a normal form transformation decomposing the dynamics into low and high frequencies with weak interactions. While the low part of the dynamics can be put under classical Birkhoff normal form, the high modes evolve according to a time dependent linear Hamiltonian system. We then control the global dynamics by using polynomial growth estimates for high modes and the preservation of Sobolev norms for the low modes. 
Our general strategy applies to any semi-linear Hamiltonian PDEs whose linear frequencies satisfy a very general non resonance condition.
The (NLW) equation on $\T^d$ is a good example  since the standard Birkhoff normal form applies only when $d=1$ while our strategy applies in any dimension. 

\end{abstract}



\maketitle

\tableofcontents

\section{Introduction}

Let us consider the nonlinear wave equation set on the $d$-dimensonal torus $\T^d = (\R / 2\pi \Z)^d$ with $d \geq 1$, 
\begin{equation}\tag{NLW}\label{nlw}u_{tt} - \Delta u + mu + f(u) =0,\quad x\in\T^d,\end{equation}
satisfied by a real valued function $u(t,x)$ 
with given initial data $u(0)\equiv u(0,\cdot)$ and $\dot u(0)\equiv \dot u(0,\cdot) = \partial_t u(0,\cdot)$. 
The function
$f \in \mathcal{C}^\infty(\R,\R)$ is at least of order 2 at the origin, {\em i.e.} $f(0) = f'(0) = 0$. 
For small and smooth initial data $u(0) \in H^s(\T^d)$ and $\dot u(0) \in H^{s-1}(\T^d)$ with large $s$, we are interested in a description of the long time behavior of $u(t)\equiv u(t,\cdot)$ solution of \eqref{nlw}. 

In dimension $d = 1$, it is known that if $\varepsilon$ measures the size of the initial data, the solution is controlled for arbitrary polynomial times with respect to $\varepsilon$, and for almost all $m$  away from zero. 
More precisely, it has been proved (see for instance \cite{BG06} and also \cite{Bam03})
\begin{theorem}(\cite{BG06}) Let $d=1$ and $r\geq2$. For almost any $m >0$  there exists $s_*(r)$ such that for $s > s_*(r) $ there exists $\varepsilon_0(r,s)>0$ such that for all $\varepsilon < \varepsilon_0$
\begin{equation}
\label{nlwd1}
\Norm{(u(0),\dot u(0))}{H^s \times H^{s-1}} \leq \varepsilon\,\quad \Longrightarrow \Norm{(u(t),\dot  u(t))}{H^s \times H^{s-1}} \leq 2 \varepsilon ,\quad t \leq \varepsilon^{-r}. 
\end{equation}
\end{theorem}
The crucial tool to obtain this result is to show that for a large set of parameters $m$, the frequencies $\omega_j = \sqrt{|j|^2 + m}$ of the linear wave operator satisfy a non-resonance condition of the following form: \\
Fix $r\geq3$, there exists $\gamma >0$ such that
for $\kb = (k_1,\cdots,k_p) \in (\Z^d)^p$, $\ellb = (\ell_1,\cdots,\ell_q) \in (\Z^d)^q$ with $p+q\leq r$ we have
\begin{equation}
\label{mu3}
\tag{{$\mu_3$}}
| \omega_{k_1} + \cdots + \omega_{k_p} - \omega_{\ell_1} - \cdots - \omega_{\ell_{q}}| \geq \frac{\gamma}{\mu_3(\kb,\ellb)^{\alpha}}, 
\end{equation}
unless $(|k_1|,\cdots,|k_p|)$ and $(|\ell_1|,\cdots,|\ell_q|)$ are equal up to a permutation,
and where $\mu_3(\kb,\ellb)$ denotes the third largest number amongst the collection $(|k_i|,|\ell_j|)_{i,j}$, and $\alpha$ depends on $r$.  \\
This condition, introduced in \cite{BG06}, allows to eliminate (or normalize) all the terms in the Hamiltonian of the perturbation (depending on $F$, a primitive of $f$ in \eqref{nlw}) involving at most two high Fourier modes via a Birkhoff normal form procedure.  On the other hand, it is known since  \cite{Bam03} that we can neglect  all the monomials involving more than three high modes (see for instance \cite{Gre07} or \cite{Bam07} for a simple presentation of these two facts). So once we have \eqref{mu3} we can expect a control of the Sobolev norms similar to \eqref{nlwd1}.\\   
Notice that  \eqref{mu3} is close from
the so-called {\em second-order} Melnikov non-resonance condition\footnote{This terminology refers to the original papers \cite{Mel65,Mel68} where similar conditions where introduced for proving the stability of low-dimensional invariant tori in Hamiltonian dynamics, and popularized in the KAM literature \cite{Mos67,Eli88, Bou97,XY01} and later while extending these results to Hamiltonian PDEs \cite{Kuk93,Pos96}.} that says in a formulation allowing comparison with \eqref{mu3}:\\
Fix $n\geq 2$, there exists $\gamma >0$ such that
for $\kb = (k_1,\cdots,k_p) \in (\Z^d)^p$, $\ellb = (\ell_1,\cdots,\ell_q) \in (\Z^d)^q$ with $|k_i|,|\ell_i|\leq n$  and  for $ (j_1,j_2) \in (\Z^d)^2$ with $|j_1|,|j_2|>n$
$$
| \omega_{k_1} + \cdots + \omega_{k_p} - \omega_{\ell_1} - \cdots - \omega_{\ell_{q}}+ \omega_{j_1} - \omega_{j_{2}}| \geq \frac{\gamma}{r^{\alpha}}, 
$$
unless  $(|k_1|,\cdots,|k_p|)$ and $(|\ell_1|,\cdots,|\ell_q|)$ are equal up to a permutation,
where $r=p+q$.  \\
We note that in the Melnikov case, the ``length" of the resonance ($r+2$) is free but the number of ``interior" modes (here $(2n+1)^d$) is fixed while it is exactly the converse in \eqref{mu3}. So the two conditions are not equivalent but similar. 

The condition \eqref{mu3} applies to many situations, including the one dimensional wave equations, the one dimensional nonlinear Schr\"odinger equation with an external potential \cite{Bam03,Bam07,BG06}, the $d$-dimensional nonlinear Schr\"odinger equation with a convolution potential \cite{BG03,BG06},  plane waves stability for non-linear Schr\"odinger equation \cite{FGL13},  wave equations on Zoll manifolds \cite{BDGS07} or  quantum harmonic oscillator on $\R^d$ \cite{GIP09}. 

The main difficulty of the higher dimensional case for \eqref{nlw} is that the frequencies do not satisfy the second-order Melnikov condition for a large set of parameters $m$, as already noted for instance in \cite{Bou95,Del09}. In fact, in dimension $d \geq 2$  \eqref{mu3} is not satisfied either: let for instance $|k_1| \geq |\ell_1| \geq \mu_3(\kb,\ellb)$ then in the left hand side of  \eqref{mu3} the quantity $\omega_{k_1} - \omega_{\ell_1}$ depends on indices that are not controlled by  $\mu_3(\kb,\ellb)$.
For $d = 1$, if $k_1$ and $\ell_1$ are large, the difference $\omega_{k_1} - \omega_{\ell_1}$ is close to be an integer and the parameter $m$ can be chosen in a large set so that \eqref{mu3} holds. However, in dimension  $d \geq 2$, $\omega_{k_1} - \omega_{\ell_1}$ describes a dense set at the scale $|k_1|$ which prevents \eqref{mu3} to hold. Actually we can only prove that for a large set of parameters $m$, 
the following condition (which is related to the so-called {\em first-order} Melnikov condition) holds
\begin{equation}
\label{mu2}
\tag{$\mu_2$}
| \omega_{k_1} + \cdots + \omega_{k_p} - \omega_{\ell_1} - \cdots - \omega_{\ell_{q}}| \geq \frac{\gamma}{\mu_2(\kb,\ellb)^{\alpha}}
\end{equation}
unless  $(|k_1|,\cdots,|k_p|)$ are $(|\ell_1|,\cdots,|\ell_q|)$ equal up to a permutation where $\mu_2(\kb,\ellb)$ is the second largest index in the multi-index $(\kb,\ellb)$.

Despite this known problem, some results have been proved concerning the existence of quasi-periodic solutions, where the loss of derivative is controlled by the use of KAM-Newton schemes, see \cite{Bou95}. 
%
Concerning the control of large open sets of solutions, some results can be found in \cite{DS06} and \cite{Del09} but the time control depends on the shape on the nonlinearity inducing restriction on the index $r$ (essentially driven by the annulation degree of the nonlinearity in $0$).

Another situation where \eqref{mu2} appears in a natural way is given by numerical discretization of Hamiltonian PDEs. For example standard splitting methods applied to wave equations in dimension $d = 1$ induces numerical resonances destroying the property \eqref{mu3} and degenerating to \eqref{mu2} even for generic time discretization parameters\footnote{The fundamental reason is that time discrete numerical schemes require the control of small divisors of the form $e^{i h \Omega(\kb,\ellb)}- 1$ instead of $\Omega(\kb,\ellb)$ as defined in \eqref{mu3}, where $h$ is the time discretization parameter. Hence numerical resonances can occurs when for instance $h (\omega_{k_1} - \omega_{\ell_1})$ is close to an arbitrary large multiple of $2\pi$ (see for instance \cite{Sha00,HLW06} for a finite dimensional analysis of symplectic integrators). }. In this case, it is however possible to control the solution by playing with the time integrator or with the space discretization, see \cite{CHL08,FGP10a,FGP10b,Fao12,FGL14}. 

In this paper, we propose a new method to overcome this difficulty by a careful examination of the normal form induced by \eqref{mu2} and a control in mixed Sobolev norm inspired by some tools used in numerical analysis. In particular it can be seen as a nonlinear extension of \cite{DF07}, here in a continuous in time setting. 

As byproduct of this method, we prove the following. 
\begin{theorem}\label{thm-NLW} Let $d \geq 2$, $r\geq2$ and $s > d/2$. For almost any $m>0$  there exists $s_* = s_*(r,s)$ and $\varepsilon_0(r,s)>0$ such that 
for all $\varepsilon < \varepsilon_0$ the solution to \eqref{nlw} satisfies
\begin{equation}
\label{nlwd2}
\Norm{(u(0),\dot u(0))}{H^{s_*} \times H^{{s_*}-1}} \leq \varepsilon\,\quad \Longrightarrow \Norm{(u(t),\dot  u(t))}{H^s \times H^{s-1}} \leq 2 \varepsilon ,\quad t \leq \varepsilon^{-r}. 
\end{equation}
\end{theorem}
In other words, \eqref{nlwd1} holds up to a finite loss of derivative in the initial condition\footnote{Historically, a similar result with loss of derivative was first given under the condition \eqref{mu3} in \cite{Bou96}. }. Concretely we can prove that $s_*(r,s)$ is of order less than $\mathcal{O}(r^5s^5)$ but this is certainly not optimal.

The previous property \eqref{nlwd2} is in  fact a corollary of a stronger abstract result proved in Theorem \ref{main}. The main idea 
is to decompose the dynamics into low and high frequencies according to some large threshold depending on $\varepsilon$,  and then to try to conjugate the system to a normal form whose dynamics can be described and controlled. 

When the \eqref{mu3} condition holds, this normal form approach allows to conjugate all the \eqref{nlw} flow to a flow preserving the $H^s$ norm  up to terms that are arbitrarily small. When only \eqref{mu2}  is satisfied, this cannot be done, and linear terms remains in the dynamics of high modes, coupled with the low modes. 
Fortunately, these terms can be put under a {\em symmetric} Hamiltonian quadratic form. Hence, despite the linear nature of their dynamics, the $L^2$ norm of the high modes can be proved to be preserved over long times in the normal form analysis. This crucial information  allows to initiate an effective decomposition between low and high modes. To prove the almost preservation of higher Sobolev norms of the high modes  we use   a sort of pseudo-differential argument (or commutator Lemma) that allows to gain one derivative (see \eqref{comut} and \eqref{Vp}).  Then for given indices $s$ and $s_0$  with $s \gg s_0 > d/2$, we can control the low modes in a Sobolev norm $H^s$ and we show a polynomial growth in time of a Sobolev norm $H^{s_0}$ of the high modes of order $\mathcal{O}(t^{s_0})$. By choosing a smoother initial condition such that $\Norm{(u(0),\dot u(0))}{H^{2s} \times H^{2s-1}} \leq \varepsilon$ we then obtain (see Theorem \ref{thm-NLW2})
\begin{equation}
\label{satie}
\Norm{u(t)_{\leq N_{\varepsilon}}}{H^s} \leq 2\varepsilon \quad \mbox{and} \quad \Norm{u(t)_{> N_{\varepsilon}}}{H^{s_0}} \leq \varepsilon^r, \quad t \leq \varepsilon^{-\frac{r}{s_0 + 1}},
\end{equation}
in a regime where $s - s_0$ is large with respect to $r$. Here $u(t)_{\leq N_{\varepsilon}}$ and $u(t)_{> N_{\varepsilon}}$ denote the low and high modes parts, according to the threshold $N_{\varepsilon} = \varepsilon^{-\frac{r}{s - s_0}}$. 
When $s$ is large, we obtain \eqref{nlwd2} after a change of indices, but Theorem \ref{thm-NLW2} (see also Theorem \ref{main}) gives more precise informations. It shows that the dynamics of the low modes preserves the super-actions\footnote{Already used in \cite{BG06} for \eqref{nlw} in dimension $d = 1$ in a periodic setting.}, {\em i.e.} the quantities $J_n(t) = \sum_{|k| = n} |u_k(t)|^2$, $n \leq N_\varepsilon$ over very long times, where $u_k(t)$, $k \in \Z^d$ denote the Fourier coefficients of $u(t)$. Such a result does not hold for high modes since the interaction between  two close large modes  cannot be eliminated but only controlled.  Theorem \ref{main} expresses the fact that the condition \eqref{mu2} -much more general than \eqref{mu3}- is enough to ensure a decoupling of the dynamics of  low and high modes for very long times. 

In the previous estimate, $s_0$ has only to satisfy the condition\footnote{Condition that could probably be refined in the critical case using Galigardo-Nirenberg inequality.} $s_0 > d/2$. It also typically corresponds to what is numerically observed for initial data taken as trigonometric polynomials\footnote{Note that in the present result, all the constants depends on the Sobolev indices $s$ and $s_0$. An optimization of these constant could be done in an analytic context, following the technics used in \cite{FG13}.} for which the dynamics does not exhibit energy exchanges over long times. Such two-stage norms with different Sobolev scaling were previously used in the context of numerical analysis of splitting methods for Schr\"odinger equations in the linear case, see \cite{DF07} and \cite{DF09} where again the preservation of the $L^2$ norm of the high-modes was crucial to obtain a global control of the dynamics.  


We also believe that results of the form \eqref{nlwd2} or \eqref{satie} involving mixed Sobolev norms provides a natural setting for numerical discretization, for which \eqref{mu2} is the generic control of non-resonance condition. It might for instance allow to weaken the usual CFL conditions required, or to derive low-order integrators  following the analysis of \cite{HS17,OS18}.

\subsection*{Acknowledgement} 
During the preparation of this work the three authors benefited from the support of the Centre Henri Lebesgue ANR-11-LABX- 0020-01 and B.G. was supported 
 by ANR -15-CE40-0001-02  ``BEKAM''  and ANR-16-CE40-0013 ``ISDEEC'' of the Agence Nationale de la Recherche. 

\section{Abstract statement}

\subsection{Hamiltonian formalism}

We recall in this subsection the formalism used in \cite{Gre07,FG13,BFG} to deal with infinite dimensional Hamiltonians and flows depending on a infinite set of symplectic variables $(q,p)  = (q_a,p_a)_{a \in \Z^d} \in \R^{\Z^d} \times \R^{\Z^d}$ equipped with the usual $\ell_s^2(\Z^d,\R^2)$ norm defined\footnote{with the usual notation $|a|^2 = a_1^2 + \cdots + a_d^2$ for $a = (a_1,\ldots,a_d) \in \Z^d$.} as 
$$
\Norm{(q,p)}{s}^2 = \sum_{a \in \Z^d} \langle a \rangle^{2s} (p_a^2 + q_a^2), \quad \langle a \rangle^2 = 1 + |a|^2. 
$$
As explained in \cite{Gre07}, for $\mathcal{U}$ an open set of $\ell_s^2$, for a function $H(q,p)$ such that $H \in \mathcal{C}^\infty(\mathcal{U},\R)$, with $\ell^2$ gradient  $\nabla_{(q,p)} H \in  \mathcal{C}^\infty(\mathcal{U},\ell_s^2)$, we can define the flow of a Hamiltonian system 
\begin{equation}
\label{Hamqp}
\forall\, a \in \Z^d, \quad \dot q_a = \frac{\partial H}{\partial p_a} (q,p), \quad \dot p_a = -  \frac{\partial H}{\partial q_a} (q,p). 
\end{equation}
To easily deal with normal form transformations, it is convenient to 
use the complex representation $(\xi_{a})_{a\in \Z^d} = (\frac{1}{\sqrt{2}}(q_a + i p_a))_{a\in \Z^d}$ in $\C^{\Z^d}$ equipped with the $\ell_s^2(\Z^d,\C)$ norm. 
Then with the notations
$$
\frac{\partial}{\partial{\xi_a}} = \frac{1}{\sqrt{2}} \left( \frac{\partial}{\partial{q_a}}   - i\frac{\partial}{\partial{p_a}}\right)
\quad\mbox{and}\quad 
\frac{\partial}{\partial{\bar \xi_a}} = \frac{1}{\sqrt{2}} \left( \frac{\partial}{\partial{q_a}}  + i\frac{\partial}{\partial{p_a}}\right), 
$$
the real Hamiltonian system is equivalent to the complex system, 
\begin{equation}\label{Eham2}
\forall\, a \in \Z^d,\quad
\displaystyle \dot \xi_a = - i\frac{\partial H}{\partial \bar \xi_a}(\xi,\bar \xi) =: (X_H(\xi,\bar \xi))_a 
\end{equation}
where $H(q,p) = H(\xi,\bar \xi) \in \R$ is a called a {\em real} Hamiltonian. The notation $X_H(\xi,\bar \xi) = (X_H(\xi,\bar \xi))_{a\in \Z^d}$ thus denote the Hamiltonian vector field associated with the Hamiltonian $H$. 
If we associate with $(\xi,\bar \xi)$ a complex function $\psi$ on  $\T^d$, through the formula 
\begin{equation}
\label{Exieta}
\psi (x)= \sum_{a\in \Z^d} \xi_{a} 
e^{i a x}, 
\end{equation}
then the Sobolev norm $\Norm{\psi}{H^s}$ is equivalent to the norm 
$$
\Norm{\xi}{s}^2 =  \sum_{a \in \mathbb{Z}^d} \langle a \rangle^{2s} |\xi_a|^2.
$$
The symplectic structure is given by 
\begin{equation}\label{Esymp}
\sum_{a \in \Z^d}{\dd q_a} \wedge \dd p_a =-i\sum_{a\in\Z^d}\dd \xi_a\wedge \dd \bar \xi_a.
\end{equation}
and the Poisson bracket in complex notation reads 
\begin{equation}\label{poisson}
\{F,G\} = - i \sum_{a \in \Z^d} \frac{\partial F}{\partial \xi_a}\frac{\partial G}{\partial \bar\xi_a} -   \frac{\partial F}{\partial \bar \xi_a}\frac{\partial G}{\partial \xi_a}.  
\end{equation}

\begin{definition}\label{def:2.1}
 For a given $s> d/2$ and a domain $\Uc$ containing $0$ in $\ell_s^2 := \ell_s^2(\Z^d,\C)$, we  denote by $\Hc_s(\Uc)$ the space of real Hamiltonians $P(\xi,\bar \xi) \in \R$ satisfying 
$$
P \in \mathcal{C}^{\infty}(\Uc,\R), \quad \mbox{and}\quad 
X_P \in \mathcal{C}^{\infty}(\Uc,\ell^2_s), 
$$
where $X_P$ is defined in \eqref{Eham2}. 
We will use the shortcut  $F\in \Hc_s$ to indicate that there exists a domain $\Uc$ containing $0$ in $\ell_s^2$ such that $F\in\Hc_s(\Uc)$.
\end{definition}
 Notice that for $F$ and $G$ in $\Hc_s$ the formula \eqref{poisson} is well defined in a neighborhood of $0$. 
With a given Hamiltonian function $H \in \Hc_s$, we associate the Hamiltonian system \eqref{Eham2}, and we can naturally defines its flow. 

\begin{proposition}\label{prop3.3} Let $s > d/2$. Any Hamiltonian in $\Hc_s$ defines a local flow in $\ell^2_s$ which is a symplectic transformation. 
\end{proposition} 

\subsection{Polynomial Hamiltonians}

To algebraically deal with polynomials depending on $(\xi,\bar \xi)$, we identify $\C^{\Z^d} \times \C^{\Z^d} \simeq 
\C^{\U_2 \times \Z^d}$ where $\U_2=\{\pm 1\}$ and use the convenient notation $(\xi,\bar \xi)= z = (z_j)_{j \in \U_2 \times \Z^d} \in \C^{\U_2 \times \Z^d}$ where
\begin{equation}
\label{Ezj}
j = (\delta,a) \in \U_2 \times \Z^d  \Longrightarrow 
\left\{
\begin{array}{rcll}
z_{j} &=& \xi_{a}& \mbox{if}\quad \delta = 1,\\[1ex]
z_j &=& \bar \xi_a & \mbox{if}\quad \delta = - 1.
\end{array}
\right.
\end{equation} 
We define the $\ell_s^2$ norm of an element $z = (\xi,\bar \xi)$ to be 
$$
\Norm{z}{s}^2 := \sum_{j \in \U_2 \times\mathbb{Z}^d} \langle j \rangle^{2s} |z_j|^2 =  2 \sum_{a \in \mathbb{Z}^d} \langle a \rangle^{2s} |\xi_a|^2  = 2\Norm{\xi}{s}^2. 
$$
 where for $j=(\delta,a) \in \U_2 \times \Z^d$ we set $ \langle j \rangle^2=\langle a \rangle^2$. 
With this notation, the Hamiltonian system \eqref{Eham2} can be written  
$$
\dot z = X_H(z), \quad \mbox{where}\quad (X_H(z))_{(\delta,a)} := - i \delta (X_H(\xi,\bar \xi))_a, \quad z = (\xi,\bar \xi)\, . 
$$
Here $X_H(\xi,\bar \xi)$ denote the vector field in \eqref{Eham2}. Another way to formulate this notation is to say that with the identification \eqref{Ezj}, the vector field $X_H(\xi,\bar \xi)$ is naturally extended as $(X_H(z))_j = - i \delta \frac{\partial H}{\partial z_{\overline j}}(z)$, for $j = (\delta,a) \in \U_2 \times \Z^d$, where $\overline j = (- \delta,a)$. 

For $\kb=  (k_1,\ldots,k_m) = (\delta_i,a_i)_{i=1}^m\in  (\U_2 \times\mathbb{Z}^d)^m$ we denote the momentum 
$$\M(\kb)=\sum_{i=1}^m \delta_ia_i, $$
and we define the conjugate $\overline{\kb} = (- \delta_i,a_i)_{i=1}^m\in  (\U_2 \times\mathbb{Z}^d)^m$. We set  
\begin{equation}
\label{eqMm}
\M_m = \{\kb \in  (\U_2 \times\mathbb{Z}^d)^m \, \mid \, \M(\kb) = 0\}
\end{equation}
the set of zero-momentum multi-indices.\\
For a given $\kb \in \M_m$ we write 
$$
z_\kb = z_{k_1} \cdots z_{k_m}. 
$$
We also define 
\begin{multline}
\label{eqRm}
\Rc_m = \{  (\delta_j,a_j)_{j=1}^m \in \M_m \ | \ \exists \sigma \in \mathfrak{S}_m,\forall j=1,\dots, m, \ \delta_j =-\delta_{\sigma_j} \\ \mathrm{and} \ \langle a_j \rangle = \langle a_{\sigma_j} \rangle    \}
\end{multline}
the set of {\bf resonant} multi-indices. Note that by construction if $m$ is odd then $\Rc_m$ is empty and that if $\kb=(\delta_j,a_j)_{j=1}^m \in \Rc_m$ is associated with a permutation $\sigma$ then we have 
\begin{equation}
\label{dansleZ}
z_{\kb} = \prod_{\substack{1\leq j\leq m\\ \delta_j = 1}} \xi_{a_j} \bar\xi_{a_{\sigma_j}}. 
\end{equation}
%
%

\begin{definition}\label{polydef}
We say that $P(z)$
 is a homogeneous polynomial of order $m$ if it can be written  
 \begin{equation}
 \label{poly}
 P(z)=P{[c]}(z) = \sum_{\jb \in \Mc_m} c_\jb z_\jb, \quad\mbox{with}\quad c=(c_\jb)_{\jb\in\Mc_m} \in \ell^\infty(\Mc_m),  
 \end{equation}
 and such that the coefficients $c_\jb$ satisfy $c_{\overline \jb} = \overline{c_{\jb}}$. 
 \end{definition}
 Note that the last condition ensures that $P$ is real valued, as the set of indices are invariant by the application $\jb \mapsto \overline \jb$.  
 Following \cite{FG13,BFG} but in a $\ell_s^2$ framework, we get the following Proposition. It turns out to be a consequence of the more general Lemma \ref{protusion} proved below. 
 \begin{proposition}
\label{polyflow}
Let  $s > d/2$.
 \begin{itemize}
 \item[(i)]A homogeneous polynomial, $P{[c]}$, of degree $m\geq2$ belongs to $ \Hc_s(\ell^2_s)$ and we have
\begin{equation}
\label{Echamp}
  \Norm{X_{P{[c]}}(z) }{s} \leq (C_{s})^m \Norm{ c}{\ell^\infty}\Norm{z}{s}^{m-1}, \quad z = (\xi,\bar \xi),
\end{equation}
for some constant $C_{s}$ depending only on $s$. \\
\item[(ii)] For two homogeneous polynomials, $P{[c]}$ and $P{[c']}$, of degree respectively $m$ and $n$, the Poisson bracket  is a homogeneous polynomial of degree $m+n-2$, $\{P{[c]},P{[c']}\}=P{[c'']} $  and we have the estimate
\begin{equation}
\Norm{ c''}{\ell^\infty} \leq 2 m n \Norm{ c}{\ell^\infty} \Norm{ c'}{\ell^\infty}
\label{Ebrack}
.\end{equation}
\end{itemize}
\end{proposition}
We end this section with a result concerning Lie transformations. We recall that to a Hamiltonian function $F$ we associate, if it exists, the Lie transformation $\Phi_{F}^1$ which is the time one flow generated by $F$. This transformation is automatically symplectic.
\begin{lemma}\label{exist} Assume that $s> d/2$ and  
let $P= P{[c_3]}+\dots+P{[c_r]}$, $c_m\in \Mc_m$, be a polynomial of order at least $3$ at the origin, then for all  $\nu\leq  \kappa_{r,s}\big(\max( \Norm{c_3}{\ell^\infty},\dots, \Norm{c_r}{\ell^\infty})\big)^{-1}$, the Lie transformation $\tau=\Phi_{P}^1$ is well defined  on a  ball  $B_s(0,\nu)$ of $\ell^2_s$ with values in $B_s(0,2\nu)$ and we have
\begin{equation}
\label{antalgiques}
 \Norm{\tau(z)-z}{s}\leq   C_{r,s} \max_m \Norm{ c_m}{\ell^\infty} \nu^{2}, \quad \forall\, z\in B_s(0,\nu),
\end{equation}
where $C_{r,s}$ and $\kappa_{r,s}$ are some positive constants depending only on $r$ and $s$.
\end{lemma}
\proof  In view of Proposition \ref{polyflow} assertion (i) we deduce from  the Cauchy-Lipschitz Theorem that the flow $\Phi_{P}^t$ is locally well defined on $\ell^2_s$. Furthermore we have for $z\in B_s(0,\nu)$ and as long as  $z(t)=\Phi_{P}^t(z)\in B(0,2\nu)$
\begin{multline*}
\Norm{\Phi_P^t(z)-z}{s}\leq \left| \int_0^t  \Norm{X_P(z(w))}{s}dw \right| \leq \sum_{m=3}^r C_{m,s} \Norm{ c_m}{\ell^\infty}|t|  (2\nu)^{m-1}\\ \leq C_{r,s} |t| \max_m \Norm{ c_m}{\ell^\infty}  \nu^{2}.
\end{multline*}
Thus we conclude by a bootstrap argument that, taking $\nu$ small enough, the flow is defined for all $z\in B(0,\nu)$ up to time $t = 1$ and that $z(t)$ remains in $B(0,2\nu)$ for $t\leq 1$ and satisfies \eqref{antalgiques}. 
\endproof

 \subsection{Non resonance condition}
\begin{definition}\label{defNR} A family of frequencies $\om=\{\om_a, a \in \mathbb{Z}^d\}$ is {\bf non resonant},  if there exist  $(\alpha(r))_{r\geq 1}\in (\mathbb{R}_+^*)^{\mathbb{N}}$ and $(\gamma_r)_{r\geq 1}\in (\mathbb{R}_+^*)^{\mathbb{N}}$, such that for all $r\geq 1$, all $N \geq 1$  and all $\kb=(\delta_i,a_i)_{i=1}^r\in  (\U_2 \times\mathbb{Z}^d)^r $ 
satisfying $\langle a_i \rangle \leq N,$ for $i=1,\cdots,r$, we have
\begin{align}
\label{H1}\tag{H1} |\delta_1\om_{a_1}+\cdots+\delta_r\om_{a_r}&|\geq \gamma_r N^{-\alpha(r)}, \quad \text{when }\kb\notin \Rc_r,\\
\label{H2}\tag{H2} |\delta_1\om_{a_1}+\cdots+\delta_r\om_{a_r}&+\om_{b}|\geq \gamma_{r+1} N^{-\alpha(r+1)},\; \forall \langle b\rangle >  N\; \text{with}\;\M(\kb)+b=0,\\
\label{H3}\tag{H3} |\delta_1\om_{a_1}+\cdots+\delta_r\om_{a_r}&+\om_{b_1}+\om_{b_2}|\geq \gamma_{r+2} N^{-\alpha(r+2)},\quad \forall \langle b_1\rangle,\langle b_2\rangle > N \\  \nonumber&\quad\quad\text{with }\quad \M(\kb)+b_1+b_2=0.
\end{align}
\end{definition}
We notice that  conditions \eqref{H1}-\eqref{H2} are equivalent to condition \eqref{mu2} introduced in the introduction while conditions \eqref{H1}-\eqref{H2}-\eqref{H3} are not equivalent to \eqref{mu3} since in \eqref{H3} we are not considering the case where the two high frequencies have opposite sign.
\begin{remark}
Note that in \eqref{H2} using the zero momentum condition (see \eqref{eqMm}), $\langle b \rangle$ is in fact bounded by $r N$. Hence \eqref{H2} is a trivial consequence of \eqref{H1}. Similarly, as in many applications $\omega_a \sim |a|^\nu$ when $a \to \infty$, for some $\nu>0$, \eqref{H3} is also a consequence of \eqref{H1} as we can restrict \eqref{H3} to a set of $(b_1,b_2)$ that are  bounded by $C(r) N$. 
\end{remark}

\subsection{Statement}

Let us start with the following notation: For $\xi \in \ell_s^2$ and a given number $N$, we decompose $\xi = \xi_{\leq N} + \xi_{>N}$ where for all $j \in \mathbb{U}_2 \times \Z^d$, 
$$
(\xi_{\leq N})_{j} = 
\left\{\begin{array}{ll}
\xi_j & \mbox{for} \quad \langle j \rangle \leq N,\\[1ex]
0 & \mbox{for} \quad \langle j \rangle > N,
\end{array}
\right.
\quad\mbox{and}\quad
\xi_{>N} = \xi - \xi_{\leq N}. 
$$ 
Given a function $\psi\in H^s(\T^d)$ with Fourier coefficients $\xi_a$, $a \in \Z^d$, and a number $N\geq 1$ the previous decomposition induces naturally the decomposition $\psi=\psi_{\leq N}+\psi_{>N}$ with
$$\psi_{\leq N}(x)=\sum_{ \langle a \rangle\leq N} \xi_a e^{i a\cdot x} \quad \text{and }\quad \psi_{>N}(x)=\sum_{ \langle a \rangle > N} \xi_a e^{i a\cdot x}.$$
Similarly, we note $z_{\geq N}$ and $z_{>N}$ the decomposition induced by the notation \eqref{Ezj}. 

We now make our hypothesis on the Hamiltonian $H$ that we will consider. 
\begin{hypothesis}
\label{beau}
The Hamiltonian $H$ can be written 
\begin{equation}
\label{H}
H = H_2 + P = \sum_{ a  \in \mathbb{Z}^d}\om_a |\xi_a |^2 + P
\end{equation}
where $ \omega = (\omega_a)_{a \in \Z^d}$ is a non resonant family of real numbers in the sense of Definition \ref{defNR}, $P$ belongs to $ \Hc_s$  for some $s>d/2$ and $P$ is of order at least $3$ at the origin (which means that $P$ and its differentials up to the order $2$ vanish at $0$).
\end{hypothesis}

Note that in general the frequencies $\omega_a$ will not be uniformly bounded with respect to $a$, and hence the quadratic part $\sum_{a\in \Z}\omega_a |\xi_a|^2$, does not belong to $\Hc_s$. Nevertheless it generates a continuous flow which maps $\ell^2_s$ into $\ell^2_s$ explicitly given for all time $t$ and for all indices $a$ by  $\xi_a(t)=e^{-i\omega_a t}\xi_a(0)$. Furthermore this flow has the group property.  By standard arguments (see for instance \cite{Caz03} and \cite{BFG} in a similar framework), this is enough to define the local symplectic flow,  $\dot z =X_H(z)$, in $\ell_s^2$.

\begin{theorem}\label{main} Let $H$ be a Hamiltonian satisfying Hypothesis \ref{beau}.
 Then for all $r\geq 2$ and all $s > s_0>d/2$  satisfying 
\begin{equation}
\label{conds} 
s - s_0 \geq  s_*(r) := 6r^2 \alpha(3r) + 2dr
\end{equation}
 there exists $\eps_0(r,s,s_0,\omega)>0$ such that for all $\eps<\eps_0(r,s,s_0,\omega)$ the solution $\xi(t)$, generated by the flow of $H$ issued from an initial datum $\xi(0) \in \ell_{2s}^2$ satisfying $\Norm{\xi(0)}{2s} \leq \eps$, exists for all time $t\leq \eps^{-\frac r{s_0+1}}$ and satisfies
\begin{equation}
\label{great}
\begin{array}{l}
\forall\, \langle a \rangle\leq N_\varepsilon :=\eps^{-\frac r{s -s_0}},\\[2ex]
\forall\, t\leq \eps^{-\frac{r}{s_0+1}}, 
\end{array}\qquad
\left|\sum_{\substack{ b \in  \mathbb{Z}^d \\   \langle b \rangle=\langle a \rangle}}|\xi_b(t)|^2-\sum_{\substack{b\in \mathbb{Z}^d    \\\langle b \rangle=\langle a \rangle}}|\xi_b(0)|^2\right|\leq \eps^3 \langle a \rangle^{-2s}  
\end{equation}
and 
$$ \left\{ \begin{array}{lll}\Norm{\xi(t)_{\leq N_{\varepsilon}}}{s}^2 = \displaystyle \sum_{ \langle a \rangle\leq N_\varepsilon }\langle a \rangle^{2s} |\xi_a(t)|^2\leq 4\eps^2\\[2ex] \Norm{\xi(t)_{> N_\varepsilon}}{s_0}^2 = \displaystyle 
\sum_{ \langle a \rangle> N_\varepsilon } \langle a \rangle^{2s_0} |\xi_a(t)|^2\leq \eps^{2r} \end{array} \right. \quad\text{for}\quad t\leq \eps^{-\frac r{s_0+1}}.$$
\end{theorem}

%

Note that by playing with the indices, we can obtain the following corollary (in the same vein as Theorem \ref{thm-NLW}) which essentially proves that arbitrary high regularity small solutions to \eqref{Eham2} are controlled over arbitrary long times. 
\begin{corollary}
\label{messiaen}
Let $H$ be a Hamiltonian satisfying Hypothesis \ref{beau}.
For all $r\geq 2$ and $s > d/2$, there exists $s_*(r,s)$ and $\varepsilon_0 (r,s,\omega)>0$ such that for all $\varepsilon < \varepsilon_0(r,s,\omega)$,  
$$
\Norm{\xi(0)}{{s_*}} \leq \varepsilon \Longrightarrow \Norm{\xi(t)}{{s}} \leq 2 \varepsilon \qquad \mathrm{for } \ t \leq \varepsilon^{-r }. 
$$
\end{corollary}
\begin{proof}
It is a consequence of the previous Theorem by replacing $r$ by $r(s_0 +1)$, $s_0$ by $s$ and assuming that $s_*$ is large enough with respect to $r$ and $s$.  
\end{proof}

\section{Application to \eqref{nlw} on $\T^d$.}

 Introducing $v=u_t\equiv\dot u$ we rewrite 
 \eqref{nlw} as 
\be\label{systnlw}
 \left\{\begin{array}{ll}
 \dot u &=  
 v,\\
 \dot v &=-\Lambda^2 u    -f(u)\,,
\end{array}\right.
\ee
where $\Lambda=(-\Delta+m)^{1/2}$ and $f \in \mathcal{C}^\infty$ having a zero of order at least two at the origin. 
When $m >0$, we can define
 \be\label{psi}
 \psi =\frac 1{\sqrt 2}(\Lambda^{1/2}u  + i\Lambda^{-1/2}v), \ee
and we get that $(u,v)\in H^s(\T^d,\R)\times H^{s-1}(\T^d,\R)$  is solution of \eqref{systnlw} if and only if $\psi\in H^{s+1/2}(\T^d,\C)$ is solution of
\be\label{nlwpsi}
i \dot \psi =\Lambda \psi+ \frac{1}{\sqrt 2}\Lambda^{-1/2}f\left(\Lambda^{-1/2}\left(\frac{\psi+
\bar\psi}{\sqrt 2}\right)\right)\,.
\ee
Then,  endowing the space   $L^2(\T^d, \C)$ with the standard (real) symplectic structure 
$\ 
-id\psi\wedge d\bar \psi 
$
  equation 
 \eqref{nlwpsi} reads as a Hamiltonian equation
$$
i\dot \psi=\frac{\partial H}{\partial \bar\psi}
$$
where $H$ is the hamiltonian function
$$
H(\psi,\bar\psi)= \frac{1}{(2\pi)^d}\int_{\T^d}(\Lambda \psi)\bar\psi \, \dd x + \frac{1}{(2\pi)^d} \int_{\T^d}F\left(\Lambda^{-1/2}\left(\frac{\psi+\bar\psi}{\sqrt 2}\right)\right)\dd x, 
$$
and $F$ is a primitive of $f$ with respect to the variable $u$, {\em i.e.} $f=\partial_u F$.\\
The linear operator $\Lambda$ is diagonal in the complex Fourier basis\footnote{Here for $a = (a_1,\ldots ,a_d) \in \Z^d$ and $x = (x_1,\ldots\, x_d) \in \T^d$ we set $a \cdot x = a_1 x_1 + \cdots + a_d x_d$} $\{ e^{i a \cdot x}\}_{a\in \Z^d}$, with eigenvalues 
\begin{equation}
\label{freq}
\om_a = \sqrt{|a|^2 + m}, \quad a = (a_1,\ldots,a_d) \in \Z^d, \quad |a|^2 = a_1^2 + \cdots a_d^2.  
\end{equation}
Decomposing $\psi$ in Fourier variables with Fourier coefficients $(\xi_a)_{a\in \Z^d}$ as in \eqref{Exieta}, 
\eqref{nlwpsi} takes the form \eqref{Eham2}
where the Hamiltonian function $H$ is given by
$
H=H_2+P
$
with,  
\begin{align*}
H_2(\xi,\bar \xi)=& \sum_{a\in\Z^d}\om_a |\xi_a|^2, \\
 P(\xi,\bar\xi) =&\frac{1}{(2\pi)^d}\int_{\T^d}F\left(\sum_{a \in \Z^d}\frac{\xi_a e^{i a \cdot x} + \bar \xi_a e^{- i a \cdot x} }{\sqrt{ 2\om_a}}\right)\dd x.\end{align*}
 As $F$ is smooth, the function $P$ is in $\Hc_s$ and we can define its flow.
Finally $(u,v)\in H^s(\T^d,\R)\times H^{s-1}(\T^d,\R)$  is a solution of \eqref{systnlw} if and only if   $\xi \in \ell^2_{s- 1/2}$ is a solution of the Hamiltonian system associated with the Hamiltonian $H$. 

In order to apply Theorem \ref{main}, we need the following result
\begin{proposition}[\cite{Del09}, Theorem 2.1.1]
\label{NR} For almost all $m\in (0,+\infty)$ the family of frequencies 
$$
 \omega_a(m) = \sqrt{|a|^2 + m}, \quad a \in \Z^d
 $$
 associated with \eqref{nlw} is non resonant in the sense of Definition \ref{defNR}.
\end{proposition}
A direct proof of this proposition can also be done by using the arguments given in \cite{Bam03,EGK16} or \cite{FGL13}. By following these proofs, one can verify that $\alpha(r)$ is of order $\mathcal{O}(r^3)$. 

As a consequence, Theorem \ref{main} applies. By scaling back to the variable $(u,v)$, we obtain
\begin{theorem}\label{thm-NLW2}
Let $f$ be a $\mathcal{C}^\infty$ function with zero of order at least $2$ at the origin. 
Then for almost all  $m \in (0,+\infty)$ and for all $r\geq2$, and all $s_0 > (d + 1)/2$, there exists $s_1(r,s_0)$ such that for all $s \geq s_1(r,s_0)$, there exists $\varepsilon_0(r,s_0,s,m)$ and for all $\varepsilon < \varepsilon_0(r,s_0,s,m)$, if $(u(0),\dot u(0))  \in H^{s}\times H^{s - 1}(\T)$ satisfies $\Norm{(u(0),\dot u(0)) }{H^{s}\times H^{s - 1}} \leq \varepsilon$, then the system \eqref{nlw} admits a solution over a time $T \geq \varepsilon^{-\frac r{s_0+1}}$, and we have 
$$ \left\{ \begin{array}{lll}\Norm{u(t)_{\leq N_{\varepsilon}}}{s}\leq 2 \eps\\[2ex] \Norm{u(t)_{> N_\varepsilon}}{s_0} \leq \eps^{r} \end{array} \right. \quad\text{for}\quad t\leq \eps^{-\frac r{s_0+1}}$$
where $N_\varepsilon :=\eps^{-\frac r{s -s_0}}$.
\end{theorem}
Using that $\alpha(r)=\mathcal{O}(r^3)$ we can verify that $s_1(r,s)\simeq (rs)^5$. Then Theorem \ref{thm-NLW} is just a reformulation of Corollary \ref{messiaen}.

\section{Normal form}

The strategy used to prove Theorem \ref{main} is to put the original Hamiltonian \eqref{H} into normal form eliminating most of interactions between the low and high frequencies. By using Taylor expansion at the origin, the Hamiltonian $H$ can be written 
\be
\label{H12} H = H_2  +\sum_{m=3}^r P_m  + R_{r+1}, \quad\mbox{with}\quad H_2 =  \sum_{ a  \in \mathbb{Z}^d}\om_a |\xi_a |^2, 
\ee
where $P_m$ is a homogeneous polynomial of degree $m$, and where $R_{r+1} \in \Hc_s$ is of order $r +1$ which means that its differentials vanish up to the order $r$. In particular, it is small in the sense that we have 
\begin{equation}\label{banff}
\Norm{X_{R_{r+1}}(z)}{s} \leq (C_{s})^r\Norm{z}{s}^{r}
\end{equation}
for some constant depending only on $s$ and for $z$ small enough in $\ell_s^2$.

For 
$\jb \in \Mc_m$ let us denote by $\mu_n(\jb)$  the $n$-th largest number amongst the collection $\langle j_i \rangle_{i = 1}^m$: 
$$
\mu_1(\jb)  \geq  \mu_2(\jb) \geq \cdots \geq  \mu_m(\jb). 
$$  
By convention, we will also set $\mu_0(\jb) = +\infty$. 

Let $N$ be a fixed number and let $H$ be a Hamiltonian satisfying Hypothesis \ref{beau}. For a given $r$ we decompose   the Hamiltonian $H$ in \eqref{H12} as follows: 
$$
 H = H_2  +\sum_{m=3}^r (P_m^{(\circ)} +  P_m^{(i)} + P_m^{(ii)} + P_m^{(iii)})  + R_{r+1}, 
$$
where for all $m$, 
\begin{itemize}
\item the polynomial $P_m^{(\circ)} = P[c^{(\circ)}_m]$, depends only of low modes:\\  
\begin{equation}
\label{lowmodes}
 \forall\jb \in \Mc_m,\quad (c^{(\circ)}_m)_{\jb} \neq  0 \quad \Longrightarrow \quad \mu_1(\jb) \leq  N. 
\end{equation}
\item $P_m^{(i)} = P[c^{(i)}_m]$ contains only one high mode:\\
\begin{equation}
\label{onemode}
 \forall\jb \in \Mc_m,\quad (c^{(i)}_m)_{\jb} \neq 0 \quad \Longrightarrow \quad  \mu_1(\jb) > N \geq \mu_2(\jb),
\end{equation}
\item $P_m^{(ii)} = P[c^{(ii)}_m]$ contains only two high modes:\\ 
\begin{equation}
\label{twomodes}
 \forall\jb \in \Mc_m,\quad (c^{(ii)}_m)_{\jb} \neq 0 \quad \Longrightarrow \quad \mu_2(\jb) > N \geq \mu_3(\jb), 
\end{equation}
\item 
 and $P_m^{(iii)} = P[c^{(iii)}_m]$ contains at least three high modes:\\
\begin{equation}
\label{threemodes}
 \forall\jb \in \Mc_m,\quad (c^{(iii)}_m)_{\jb} \neq 0
\quad \Longrightarrow \quad \mu_3(\jb) >  N. 
\end{equation}
\end{itemize}
Following \cite{BG06}, we know that polynomials of the form $P_m^{(iii)} = P[c^{(iii)}_m]$ are already small in the sense that $$\Norm{X_{P_m^{(iii)}}}{s}\leq CN^{-s}\|z\|_s^{m-1}.$$
On the other hand, thanks to our non resonance condition,  polynomials of the form $P_m^{(i)} = P[c^{(i)}_m]$ can be killed by a symplectic change of variables since it cannot be resonant (only one high mode) and polynomials of the form $P_m^{(\circ)} = P[c^{(\circ)}_m]$ can be normalized by a standard Birkhoff normal form procedure. In addition to these two known facts, the following Theorem says that we can also symmetrize polynomials of the form $P_m^{(ii)} = P[c^{(ii)}_m]$:

\begin{theorem}
\label{normalform}
Assume that the frequencies $\omega = (\omega_a)_{a \in \Z^d}$ are non resonant in the sense of Definition \ref{defNR}, and let $r\geq2$ be given.  There exists a constant $C$ depending on $r$, such that for all $N\geq1$, there exists a polynomial Hamiltonian 
\begin{equation}
\label{chi}
\chi = \sum_{m=3}^r (\chi_m^{(\circ)} +  \chi_m^{(i)} + \chi_m^{(ii)}) 
\end{equation}
such that $\chi_m^{(\circ)} = P[a_m^{(\circ)}]$, $\chi_m^{(i)} = P[a_m^{(i)}]$, $\chi_m^{(ii)} = P[a_m^{(ii)}]$, contain zero, one and two high modes respectively (i.e. satisfy \eqref{lowmodes}, \eqref{onemode} and  \eqref{twomodes} respectively), with coefficients satisfying 
\begin{equation}
\label{masmoudi}
\Norm{a_m^{(\circ)}}{\ell^\infty} + \Norm{a_m^{(i)}}{\ell^\infty} + \Norm{a_m^{(ii)}}{\ell^\infty} \leq C N^{r \alpha(r)},
\end{equation}
and such that the Lie transformation $\Phi_{\chi}^1$, whose existence is guaranteed in a neighborhood of the origin of $\ell^2_s$ for all $s>d/2$ by Lemma \ref{exist}, puts $H$ in normal form:
\begin{equation}
\label{sansdec}
H \circ \Phi_{\chi}^1 = H_2  +\sum_{m=3}^r (Z_m^{(\circ)} +  S_m^{(ii)} + \tilde P_m^{(iii)})  + \tilde R_{r+1}, 
\end{equation}
where 
$
Z_m^{(\circ)} = P[b_m^{(\circ)}]$, $S_m^{(ii)} = P[b_m^{(ii)}]$, $\tilde P_m^{(iii)} = P[b_m^{(iii)}]$, contain zero, two and at least three high modes respectively (i.e. satisfy \eqref{lowmodes}, \eqref{twomodes}, \eqref{threemodes}), with coefficients satisfying 
$$
\Norm{b_m^{(\circ)}}{\ell^\infty} + \Norm{b_m^{(ii)}}{\ell^\infty} + \Norm{b_m^{(iii)}}{\ell^\infty} \leq C N^{r \alpha(r)}. 
$$
Moreover \begin{itemize}
\item $Z_m^{(\circ)}$ contains only resonant monomials, which means that 
$$
\forall\, \jb \in \Mc_m, \quad 
\jb \notin \Rc_m  \Longrightarrow \quad (b^{(\circ)}_m)_{\jb} = 0. 
$$
\item $S_m^{(ii)}$ contains terms that are symmetric in the high modes which means that if $b_\jb^{(ii)} \neq 0$ for $\jb \in \Mc_m$, the two highest modes are of opposite signs: they are of the form $(\delta,a)$ and $(-\delta,b)$ for some $a$ and $b \in \Z^d$.

\item The remainder term $\tilde R_{r+1}$ is of the form
\begin{equation}
\label{last?}
\tilde R_{r+1} = R_{r+1} \circ \Phi_{\chi}^1 +  \int_0^1 (1 - s)^{r+1}P[b_{r+1}]\circ \Phi_\chi^s \dd s
\end{equation}
 where $P[b_{r+1}]$ defines a homogeneous polynomial of order $r+1$ with coefficients bounded by $\Norm{b_{r+1}}{\ell^\infty} \leq C N^{r\alpha(r)}$. 

\end{itemize}
\end{theorem}

\proof
The proof is standard and use the non resonant Birkhoff normal form procedure (see \cite{BG06,BDGS07,Gre07}). We follow here the construction made in \cite{FG13}. 
By using the formal series expansions $H = H_2 + \sum_{m \geq 3}P_m$ and $\chi = \sum_{m \geq 3} \chi_m$ in homogeneous polynomials, the formal normal form problem is to find $\chi$ and $X$ (under normal form) such that 
$$
H \circ \Phi_\chi^{1} = \sum_{k \geq 0} {\rm ad}_{\chi}^k (H_2 + P) = H_2 + X. 
$$
In the formal series algebra, this problem is equivalent to a sequence of homological equation of the form
$$
\forall\, m \geq 3, \quad 
\{ H_2,\chi_m\} = Q_m - X_m, 
$$
where $Q_m$ depends on the function $P_k$, $\chi_k$ and $X_{k}$, $k < m$ previously constructed.  It is obtained by iterated Poisson brackets preserving the homogeneity of polynomial and boundedness of coefficients (see \eqref{Ebrack}). Formula for $Q_m$ can be found in \cite{FG13}, Eq. (3.4). 

Now assume that $Q_m = \sum_{\jb \in \Mc_m} q_\jb z_\jb$ is given. For a given $N$ we can decompose it into terms containing zero, one, two and at least three high modes: $Q_m = Q_m^{(\circ)} +  Q_m^{(i)} + Q_m^{(ii)} + Q_m^{(iii)}$ and a similar decomposition for the coefficients $q_\jb$. The normal form term $X_m$ is then the sum of the resonant terms in $Q_m^{(\circ)}$ (contributing to the term $Z^{(\circ)}_m$), the symmetric part of $Q_m^{(ii)}$ (contributing to the term $S_m^{(ii)}$), and the term $Q_m^{(iii)}$ (contributing to the term $\tilde P^{(iii)}_m$). By noting that 
$$
\{H_2,z_\kb\} = i \Omega(\kb) z_\kb \quad \mbox{with} \quad 
\Omega(\kb) = - \sum_{i = 1}^m \delta_i \omega_{a_i}, \quad \kb = ((\delta_i,a_i))_{i = 1}^m \in \Mc_m, 
$$
we then solve the other terms by setting 
$$
\chi_m = \sum_{\jb \in \Mc_m\backslash\Rc_m} \frac{q_\jb^{(\circ)}}{i \Omega(\jb)} z_\jb + \sum_{\jb \in \Mc_m} \frac{q_\jb^{(i)}}{i \Omega(\jb)} z_\jb + \sum_{\jb \in \Mc_m\backslash\Sc_m} \frac{q_\jb^{(ii)}}{i \Omega(\jb)} z_\jb, 
$$
where $\Sc_m$ denote the set of indices with two symmetric high modes. Note that when $\jb \in \Mc_m \backslash \Sc_m$, the two highest modes (larger that $N$) have the same sign, and the denominator $\Omega(\jb)$ is controlled by \eqref{H3}. Similarly, the first term can be controlled using \eqref{H1} and the second using \eqref{H2}. We then observe that we loose a factor $N^{\alpha(m)}$ after each solution of the homological equation, yielding a bound of order $N^{r\alpha(r)}$ after $r$ iterations. Note also that all the operations (solution of the Homological equation and Poisson brackets) preserve homogeneity and the reality of the global Hamiltonians. 

It is easy to see that for all $s > s_0$ for $z$ small enough (such that $C_{r,s} N^{r \alpha(r)} \Norm{z}{s} \leq 1$ for some constant $C_{r,s}$ by using \eqref{Echamp}) the flow $\Phi_\chi^1(z)$ is well defined and locally invertible in $\ell_s^2$ (its inverse being $\Phi_\chi^{-1}$), see Lemma \ref{exist}. 

Finally to obtain \eqref{sansdec} we use a Taylor expansion of the term $H \circ \Phi_\chi^t$ for $t \in (0,1)$. \endproof

For $\xi \in \ell_s^2$, 
we define the {\it pseudo-actions}:
\begin{equation}
\label{gigi}
J_a(\xi,\bar \xi) =\sum_{\substack{b \in \Z^d \\ \langle a\rangle = \langle b\rangle}} |\xi_b|^2 \qquad (\mbox{and} \quad J_j = \frac12 \sum_{\substack{ \ell\in \U_2 \times\mathbb{Z}^d \\    \langle \ell \rangle=\langle j \rangle}}|z_\ell|^2, \quad j \in \U_2 \times\mathbb{Z}^d).
\end{equation} 
By definition of the resonant set $\Rc_m$ (see \eqref{eqRm}) and the corresponding resonant monomials (see \eqref{dansleZ}) we see that for all $m$, the normal form terms $Z_m^{(\circ)}$ can be written\footnote{Recall that $Z_{m}^{(\circ)} = 0$ when $m$ is odd.} 
$$
Z_m^{(\circ)}(\xi,\bar \xi)  = \sum_{\substack{\kb \in \mathcal{R}_{m} \\ \forall i, \langle k_i \rangle \leq N}} c_\kb z_\kb = \sum_{\substack{\ab,\bb \in (\mathbb{Z}^d)^\frac{m}{2}, \\
 			  \forall i,  \langle a_i \rangle = \langle b_i \rangle\leq N}} c_{\ab\bb} \prod_{i=1}^{\frac{m}{2}} \xi_{a_i} \bar \xi_{b_i}, \quad z = (\xi,\bar \xi).
$$
We notice that a polynomials in normal form commutes with all the {\it pseudo-actions}:
\begin{equation}
\label{commuteboy}
\{Z_m^{(\circ)},J_a\}=0, \quad  \forall a \in \mathbb{Z},
\end{equation}
in such a way the flow generated by $Z_m^{(\circ)}$ will not modify the $\ell_s^2$ norms since $\Norm{z}{s}^2=\sum (1+a^2)^{2s}J_a$.
\section{Proof of the main Theorem}

To prove the main Theorem, we will need the following Lemma which controls polynomial vector fields in mixed Sobolev norms, as high and low modes are not controled at the same Sobolev scale. 
\begin{lemma}
\label{protusion}
For all $s \geq s_0>d/2$, there exists a constant $C$ such that 
for all $m \geq 3$, for all $N\geq1$ and for all $c \in \ell^\infty(\Mc_m)$, if the polynomials $P[c]$ contains at least $n$-th high modes, $n = 0,1,2,3$ \textrm{(}{\em i.e.} $c_\jb \neq 0 \Longrightarrow \mu_{n}(\jb)  > N$\textrm{)} then  we have for $z = (\xi,\bar \xi)$ with $z_{\leq N}\in\ell_s^2$ and $z_{ > N}\in \ell^2_{s_0}$
\begin{equation}
\label{belami}
\Norm{X_{P[c]}(z)_{\leq N}}{s} \leq 
\left\{
\begin{array}{ll}
C^m\Norm{c}{\ell^\infty}\Norm{z_{\leq N}}{s}^{m-1}& \mbox{if} \quad n = 0 \\[2ex]
C^m\Norm{c}{\ell^\infty} N^{n(s-s_0)}  \Norm{z_{\leq N}}{s}^{m-n - 1} \Norm{z_{ > N}}{s_0}^{n} &\mbox{if}\quad n \geq 1,
\end{array}
\right.
\end{equation}
and 
\begin{equation}
\label{horla}
\Norm{X_{P[c]}(z)_{> N}}{s_0} 
\leq
\left\{
\begin{array}{ll}
0 & \mbox{if}\quad n = 0, \\[2ex]
  C^m\Norm{c}{\ell^\infty}  N^{s_0 - s}\Norm{z_{\leq  N}}{s}^{m-n - 1 }& \mbox{if} \quad n = 1, \\[2ex]
C^m\Norm{c}{\ell^\infty}\Norm{z_{\leq  N}}{s}^{m-n }  \Norm{z_{> N}}{s_0}^{n - 1}&\mbox{if}\quad n \geq 2.
\end{array}
\right.
\end{equation}
\end{lemma}
\proof
For $z = (\xi,\bar \xi)\in \ell_s^2$, we have 
$$
\Norm{X_{P[c]}(z)_{\leq N}}{s}^2 \leq \sum_{\langle \ell\rangle\leq N}\langle \ell\rangle^{2s}\left|\frac{\partial P[c]}{\partial z_\ell}(z) \right|^2 
\leq  m^2\Norm{c}{\ell^\infty}^2\sum_{\langle \ell\rangle\leq N}\langle \ell\rangle^{2s_0} P_\ell(z)^2,
$$
where 
\begin{equation}
\label{Pell}
P_\ell(z)= \langle \ell\rangle^{s - s_0} \sum_{\substack{\M(\jb)=0 \\ \mu_n(\jb,\ell) > N}}|z_{j_1}|\cdots| z_{j_{m-1}}|. 
\end{equation}
Let $\jb=(j_1,\dots,j_m)$ be given such that 
$\M(\jb,\ell) = 0$. We have 
$$
\label{hernie}
\langle \ell\rangle^{s - s_0} \leq  ( \langle j_1\rangle+\cdots +\langle j_{m-1}\rangle)^{s - s_0} \leq m^{s - s_0} \langle j_1\rangle^{s- s_0}\cdots  \langle j_{m-1}\rangle^{s-s_0}. 
$$
If $P[c]$ contains no high modes, {\em i.e.} if $n = 0$ then we can define $\tilde z_{j} = \langle j \rangle^{s- s_0}z_j$ and we obtain 
$$
P_\ell(z) \leq \sum_{\substack{\M(\jb,\ell)=0 \\†\mu_1(\jb) \leq N}}|\tilde z_{j_1}|\cdots| \tilde z_{j_{m-1}}|. 
$$
For a function $g$ on the torus, we will denote   $\hat g_b=(1/2\pi)^{d}\int_{\T^d}g(x)e^{-ib \cdot x}\dd x$, $b \in \Z^d$, the Fourier coefficients of $g$.
Then denoting
$f(x)=\sum_{\substack{\langle \ell \rangle \leq N}}|\tilde z_\ell| e^{i\delta\, a\cdot x}$
we have 
$$\sum_{\substack{\M(\jb,\ell)=0 \\ \mu_1(\jb) \leq N}}|\tilde z_{j_1}| \cdots| \tilde z_{j_{m-1}}| =  (\widehat{f^{m-1}})_{- \delta a}, \quad \ell = (\delta,a)\in \U_2 \times \Z^d $$
for some constant $C$ depending on $m$. So we get for some generic constant $A$ depending on $s$ and $s_0$ but independent on $m$
\begin{multline*}
\Norm{X_{P[c]}(z)_{\leq N}}{s}^2\leq A m^{s - s_0 + 2}\Norm{c}{\ell^{\infty}}^2 \sum_{a \in \Z^d}  \langle a\rangle^{2s_0}|(f^{m-1})_a|^2 \\
=  A m^{s - s_0 + 2} \Norm{c}{\ell^{\infty}}^2\| f^{m-1}\|_{H^{s_0}}^2 
\leq  A C^m \Norm{c}{\ell^{\infty}}^2 \| f\|_{H^{s_0}}^{2m - 2},
\end{multline*}
where $\Norm{g }{H^{s_0}}$, $s_0 > d/2$ is the usual Sobolev norm on $\T^d$ equivalent to the $\ell_{s_0}^2$ norm of the Fourier coefficients of $g$. Here $C$ depends on $s$ and $s_0$ but not on $m$. 
We then note that $ \Norm{ f}{H^{s_0}}=\Norm{z_{\leq N}}{s}$ which shows the first equation in \eqref{belami}. 

To prove the second in the case $n \geq 1$, we simply bound $\langle \ell\rangle^{s-s_0}
$ by $N^{s-s_0}$ and we obtain 
$$
P_\ell(y ) \leq N^{s - s_0}  \sum_{\substack{\M(\jb,\ell)=0 \\ \mu_n(\jb) > N}}|z_{j_1}|\cdots| z_{j_{m-1}}|  = ( \widehat{f^{m - n} g^n})_{- \delta a}
$$
for $ \ell = - \delta a$
with the same notation as before, where
$$
f(x)=\sum_{†\langle \ell \rangle \leq N}|z_\ell| e^{i\delta\, a\cdot x}\quad\mbox{and}\quad g(x)=\sum_{†\langle \ell \rangle > N}|z_\ell| e^{i\delta\, a\cdot x},
$$
and we conclude as in the previous case. 

To show \eqref{horla} we use 
$$
\Norm{X_{P[c]}(z)_{> N}}{s_0}^2 \leq m^2 \Norm{c}{\ell^\infty}^2\sum_{\langle \ell\rangle> N}\langle \ell\rangle^{2s_0} P_\ell(z)^2,
$$
where
\begin{equation}
\label{Pell2}
P_\ell(z)=  \sum_{\substack{\M(\jb)=0 \\ \mu_n(\jb,\ell) > N}}|z_{j_1}|\cdots| z_{j_{m-1}}|,
\end{equation}
In the case $n = 1$, we have no high mode in the sum. However due to the zero momentum condition, there exists one mode greater than $N/(2m)$, hence with the same notation as before, we have 
$$
P_\ell(z) \leq  A N^{s_0 - s}\sum_{\substack{\M(\jb)=0 \\ \mu_n(\jb,\ell) > N}}|\tilde z_{j_1}|\cdots| \tilde z_{j_{m-1}}|,
$$
and we conclude as before. The proof of the last estimate in \eqref{horla} is similar as the proof of \eqref{belami} for $n\geq 1$ . 
\endproof

\noindent {\bf Proof of Theorem \ref{main}.} 
Let $\varepsilon > 0$, $s \geq s_0 > d/2$ and $r$ be given, and $z(0) \in \ell_{2s}^2$ such that $\Norm{z(0)}{2s} \leq \varepsilon$. 
We apply the normal form Theorem \eqref{normalform} at the order $3r$ and set 
$$N =N_\varepsilon=\eps^{-\frac r{s- s_0}}. $$
Under the hypothesis \eqref{conds}, we have the bound 
\begin{equation}
\label{CFL2}
N^{3r\alpha(3r) +d}\leq \eps^{-1/2}.
\end{equation}
In particular, all the coefficients of the polynomials Hamiltonian $\chi$, $Z_m^{\circ}$, $S_m^{(ii)}$ and $\tilde P_m^{(iii)}$ in \eqref{sansdec} are bounded by $C N^{3r\alpha(3r)} \leq C \varepsilon^{-\frac12}$. 
Hence using Lemma \ref{exist} we have
 for $\varepsilon$ small enough (in a way depending on $r$ and $s$)  that $y(0)  = \Phi_{\chi}^{-1}(z(0))$ is in $\ell_{2s}^2$ and satisfies $\Norm{y(0)}{2s} \leq \frac{5}{4}\varepsilon$. 

This implies in particular that $\Norm{y(0)_{\leq N}}{s} \leq \frac{5}{4}\varepsilon$ and
\begin{multline*}
\Norm{y(0)_{> N}}{s_0}^2  =\sum_{\langle j \rangle > N} \langle j \rangle^{2s_0} |y_j|^2 \leq N^{2s_0 - 4s }\sum_{\langle j \rangle > N} \langle j \rangle^{4s} |y_j|^2 \\ 
\leq  4\varepsilon^2 N^{4(s_0 - s)} \leq 4\varepsilon^{4r + 2}. 
\end{multline*}
Now we need to control the dynamics of $y(t)$ the solution of the Hamiltonian system associated with the Hamiltonian $H\circ  \Phi_{\chi}^1$ given by \eqref{sansdec}.
We define 
\begin{equation}
\label{boot}
T:= \sup\{ t>0\mid  \Norm{y(t)_{\leq N}}{s}\leq \textstyle\frac32\eps \quad \mbox{and} \quad 
\Norm{y(t)_{> N}}{s_0}\leq \eps^{r + 1}
\}.
\end{equation}
We notice that  $\Norm{y(0)}{2s} \leq \frac{5}{4}\varepsilon<\textstyle\frac32\eps$ and $\Norm{y(0)_{>N}}{s_0} \leq 2\varepsilon^{2r + 1}<\eps^{r + 1}$. Therefore by classical results for the definition of mild solutions of semi-linear problems in Sobolev spaces with index greater than $d/2$, $T$ is strictly positive. \\
Let us prove that if $t\leq \min(T,\eps^{-\frac r{s_0+1}})$ then $\Norm{y(t)_{\leq N}}{s}\leq \textstyle\frac{11}8\eps$ and $\Norm{y(t)_{> N}}{s}\leq \frac12 \textstyle\eps^{r+ 1}$. We will then conclude by a continuity argument that  $T\geq \eps^{-\frac r{s_0+1}}$. 

\medskip 
\noindent{\uline{{\em Control of the transformation}}}. 
In view of Lemma \ref{protusion}, under bootstrap hypothesis, the vector field $X_\chi(y(t))$ satisfies the estimates 
$$
\Norm{X_\chi(y)_{\leq N}}{s} \leq C N^{3r \alpha(3r)}( \varepsilon^2 + N^{s - s_0} \varepsilon^{r + 1} \varepsilon) \leq C \varepsilon^{\frac{3}{2}}
$$
where we used $N^{s - s_0} = \varepsilon^{-r}$, and 
$$
\Norm{X_\chi(y)_{\leq N}}{s_0} \leq C N^{3r \alpha(3r)}( N^{s_0- s} \varepsilon^2 +  \varepsilon^{r+1} \varepsilon) \leq C \varepsilon^{r + \frac{3}{2}}. 
$$
Hence this shows that as soon as $y$ satisfies the bootstrap hypothesis, i.e.
 \begin{equation}\label{OhYes}\Norm{y(t)_{\leq N}}{s}\leq \textstyle\frac32\eps\quad \text{and }\quad 
\Norm{y(t)_{> N}}{s_0}\leq \eps^{r + 1},\end{equation}
 we have for $\varepsilon$ small enough and for all $w \in (0,1)$
\begin{equation}
\label{tau}
\Norm{\Phi_\chi^w(y)_{\leq N}}{s} \leq 2 \varepsilon \quad \mbox{and} \quad \Norm{\Phi_\chi^w(y)_{> N}}{s_0} \leq 2\varepsilon^{r+1}. 
\end{equation}

Now let us write the Hamiltonian \eqref{sansdec} as $ H \circ \Phi_{\chi}^1 = H_2  +Z^{(\circ)} +  S^{(ii)} + \tilde P^{(iii)} + \tilde R_{3r+1}$ by gathering together the terms with different homogeneity. 

\medskip
\noindent \uline{{\em Control of the low modes}} $y(t)_{\leq N}$. For $j \in \mathbb{U}_2 \times \Z^d$, let $J_j (t) = J_j(y(t))$ with the definition \eqref{gigi}. 
As $Z^{(\circ)} $ and $H_2$ commute with $J_a$ for all $a \in \Z^d$, we have for $t\leq \min(T,\eps^{-\frac r{s_0+1}})$ and $\langle j \rangle \leq N$, 
$$
\langle j \rangle^{2s}|\dot J_j|=\langle j \rangle^{2s} |\{J_j, \tilde H\}|\leq  \sum_{\langle \ell\rangle=\langle j\rangle} \langle \ell \rangle^{s}\left|\frac{\partial (S^{(ii)} + \tilde P^{(iii)} + \tilde R_{3r+1})}{\partial y_{\bar\ell}}(z)\right| \langle \ell\rangle^{s} |y_\ell|. 
$$
By summing in $\langle j \rangle \leq N$, and using Cauchy-Schwarz inequality, we obtain 
$$
\sum_{\langle j \rangle \leq N} 
\langle j \rangle^{2s}|\dot J_j| \leq \Norm{ X_{S^{(ii)} + \tilde P^{(iii)} + \tilde R_{2r+1}}(y)_{ \leq N}}{s} \Norm{y_{\leq N}}{s} 
$$
Now under bootstrap hypothesis \eqref{OhYes}, we conclude using Lemma \ref{protusion}
$$
\Norm{X_{S^{(ii)} + \tilde P^{(iii)}}(y)_{\leq N}}{s} \leq C N^{3r \alpha(3r)} N^{s - s_0}\varepsilon^{2r + 2} \leq C \varepsilon^{r+\frac{3}{2}}.
$$
On the other hand since $R_{3r+1}$ is a Hamiltonian in $\mathcal H_s$ of order at least $3r+1$,   we have using  \eqref{banff}, \eqref{last?} and Lemma \ref{exist}
$$
\Norm{ X_{\tilde R_{3r+1}}(y)_{ \leq N}}{s} \leq C  N^{3r \alpha(3r)} \varepsilon^{3r}. 
$$
Hence we obtain
$$
\sum_{\langle j \rangle \leq N} 
\langle j \rangle^{2s}|\dot J_j| \leq C ( \varepsilon^{3r + 1} + \varepsilon^{r + \frac{5}{2}}).  
$$
Therefore for $t\leq \min(T,\eps^{-r + 1})$ we conclude 
\begin{align*}\sum_{\langle j \rangle \leq N}\langle j \rangle^{2s}| J_j(t)-J_j(0) |\leq \eps^\frac{7}{2} \end{align*}
for $\eps$ small enough which in turn implies that 
\begin{equation}\label{OhYes1}
\Norm{y(t)_{\leq N}}{s}\leq \textstyle\frac{11}8\eps.
\end{equation}
Furthermore by using estimates on the vector field $X_\chi$, we deduce that $z(t) = \Phi_{\chi}^1(y(t))$ satisfies \eqref{great}. 
\medskip

\noindent
\uline{{\em Control of the high modes}}.
The Hamiltonian $S^{(ii)}$ can be written as 
$$
S_m^{(ii)} = \sum_{\substack{a,b \in \Z^d \\\langle a \rangle > N, \, \langle b \rangle > N}}  B_{ab}(y_{\leq N})\xi_a \bar\xi_b, \quad y = (\xi,\bar \xi). 
$$
As the Hamiltonian is real we have $\overline{B_{ab}(y_{\leq N})} = B_{ba}(y_{\leq N})$, i.e. the operator $A = (B_{ab}(y_{\leq N})_{\langle a \rangle, \langle b \rangle > N}$ acting on $\ell_{s}^2(\Z^d_{>N})$ is Hermitian.
Moreover, we have 
\begin{equation}\label{BB}
B_{ab}(y_{\leq N}) =  \sum_{m = 1}^{2r - 2} \sum_{\substack{\jb \in \Mc_m \\\mu_1(\jb) \leq N\\\Mc(\jb) = a - b}} b_{ab, \jb} y_{\jb}
\end{equation}
where the coefficients $b_{ab,\jb}$ are uniformly bounded by $C N^{3r\alpha(3r)}$. Hence for $s > d/2$ and as soon as $y$ satisfies the bootstrap assumption \eqref{OhYes},  we obtain that 
\begin{equation}
\label{Kab}
\forall\, \langle a \rangle,  \langle b \rangle > N \quad |B_{ab}(y_{\leq N})| \leq C N^{3r\alpha(3r) } \Norm{y_{\leq N}}{s}.  
\end{equation}
When writing the dynamics of $y = (\xi,\bar \xi)$, we get 
$$
\dot \xi_a= -i \omega_a \xi_a- i \sum_{b \in \Z^d}B_{ab}(y_{\leq N})\xi_b - i Q_a(y),\qquad \langle a \rangle > N$$
where $Q_a = \frac{\partial}{\partial \bar \xi_a} (\tilde P^{(iii)}  + \tilde R_{3r +1})$. 
Using \eqref{horla} we have
\begin{equation}
\label{azerty}
\Norm{Q(y)_{> N}}{s_0}\leq C N^{3r \alpha(3r)} ( \varepsilon^{2r + 2} \varepsilon + \eps^{3r}).
\end{equation}
So using the fact that the operator $A$ is hermitian, we get for $t\leq T$ (recall that we assume that $\Norm{y(0)_{> N}}{0}\leq \eps^{2r +1}$)
\begin{equation}\label{V0}\Norm{y(t)_{> N}}{0}=2\Norm{\xi(t)_{>N}}{0}\leq 2\Norm{\xi(0)_{> N}}{0} +tC\eps^{2r + \frac{3}{2}}
 \leq C(1+t)\eps^{2r + 1}.\end{equation}

Let us define $D^{s_0}$ the diagonal operator from $\ell_{s}^2(\Z^d_{>N})$ into $\ell_{s-2s_0}^2(\Z^d_{> N})$ given by
$$D^{s_0}=\diag (\langle a \rangle^{2s_0}, \langle a \rangle> N).$$
With the notation $(\eta,\xi)_{>N} = \sum_{\langle a \rangle > N} \bar \eta_a \xi_a$, we have  
$$
\frac{\dd}{\dd t}(\xi, D^{s_0} \xi)_{> N}=- i(\xi ,[D^{s_0},A(t)]\xi)_{> N} + \mathrm{Im}  (Q(y) , D^{s_0}\xi)_{> N}
$$
where $[A,B]$ denotes the commutator of $A$ and $B$: $[A,B] = AB - BA$. \\
Hence by bootstrap hypothesis (and using \eqref{azerty})
$$
\left|\frac{\dd}{\dd t}(\xi, D^{s_0} \xi)_{>N}\right| \leq 
\big|(\xi,[D^{s_0},A(t)]\xi)_{>N}\big| +
C\eps^{3r + \frac{5}{2}}.
$$ 
Then we note that using the zero momentum condition \eqref{BB} we have $ B_{ab}(t) = 0$ when $|a - b| > 3rN$. Hence, since $s_0>1/2$, we have 
\begin{multline*} \big|(\xi,[D^{s_0},A(t)]\xi)_{>N}\big| =\Big|\sum_{\substack{\langle a \rangle,\langle b \rangle\geq N\\ |a-b|\leq 3rN}}  B_{ab}(y)(\langle a \rangle^{2s_0}-\langle b \rangle^{2s_0})\bar \xi_a\xi_b\Big|\\
\leq \sum_{\substack{\langle a \rangle,\langle b \rangle\geq N\\ |a-b|\leq 3 rN}} 2r|B_{ab}(t)||\langle a \rangle-\langle b \rangle|(\langle a \rangle^{2s_0-1}+\langle b \rangle^{2s_0-1})|\bar \xi_a\xi_b|. 
\end{multline*}
Using the bound \eqref{Kab} and the fact that $|\langle a \rangle-\langle b \rangle| \leq \langle a - b \rangle \leq CN$,  we get
\begin{align}\label{comut}
\big|(\xi,[D^{s_0},A]\xi)\big| \leq  C N^{3r\alpha(3r)} \varepsilon \sum_{\substack{\langle a \rangle,\langle b \rangle\geq N\\ |a-b|\leq 3 rN}}\langle a \rangle^{s_0-1}\langle b \rangle^{s_0}|\bar \xi_a\xi_b|    
\end{align}
where we used that for $|a-b|\leq 3rN$ we have $\langle a \rangle\leq \langle b \rangle+3 rN$.
Now we apply the Cauchy-Schwarz inequality to get
\begin{multline*}
\sum_{\substack{\langle a \rangle,\langle b \rangle\geq N\\ |a-b|\leq 3rN}}\langle a \rangle^{s_0-1}\langle b \rangle^{s_0}|\bar\xi_a\xi_b| \leq \\
\left(  \sum_{\substack{\langle a \rangle,\langle b \rangle\geq N\\ |a-b|\leq 3rN}}\langle a \rangle^{2s_0-2} |\bar \xi_a|^2\right)^{\frac12}
\left(\sum_{\substack{\langle a \rangle,\langle b \rangle\geq N\\ |a-b|\leq 3rN}}\langle b \rangle^{2s_0}|\xi_b|^2\right)^{\frac12}\\
\leq C N^d \Norm{y_{> N}}{s_0-1} \Norm{y_{>N}}{s_0}
\end{multline*}
where we used again that for $|a-b|\leq 3rN$ we have $\langle a \rangle\leq \langle b \rangle+3rN$ with $a \in \Z^d$.
Thus, using \eqref{CFL2}, we obtain 
\begin{align}\label{Vp} \frac{\dd}{\dd t}\Norm{y(t)_{> N}}{s_0}^2\leq C \Norm{y(t)_{> N}}{s_0-1} \Norm{y(t)_{> N}}{s_0}+ C\eps^{3r + \frac{5}{2}}. 
\end{align}
Then, since $s_0>1$, applying the H\"older inequality in \eqref{Vp}, we get
\begin{align}\label{merci_Holder} \frac{\dd}{\dd t}\Norm{y(t)_{> N}}{s_0}^2\leq C \Norm{y(t)_{> N}}{0}^{1/s_0} \Norm{y(t)_{> N}}{s_0}^{2(1-1/s_0)}+ C\eps^{3r + \frac{5}{2}}. 
\end{align}
Using, in \eqref{merci_Holder}, the bound on $\Norm{y(t)_{> N}}{0}$ obtained in \eqref{V0}, we get
\begin{align}\label{ready_for_the_petit_lemme} \frac{\dd}{\dd t}\Norm{y(t)_{> N}}{s_0}^2\leq C (1+t)^{1/s_0} \eps^{\frac{2r + 1}{s_0}} \Norm{y(t)_{> N}}{s_0}^{2(1-\frac1{s_0})}+ C\eps^{3r + \frac{5}{2}}. 
\end{align}

We are going to apply the following elementary lemma 
\begin{lemma}
\label{petit_mais_joli}
Let $\alpha\in (0,1)$, $f:\R \to \R_+$ a continuous function, and $x:\R \to \R_+$ a differentiable function satisfying the inequality
$$
\forall\, t \in \R, \quad  \frac{\dd}{\dd t} x(t) \leq \frac1{1-\alpha}  f(t) (x(t))^{\alpha}.
$$
Then we have the estimate
$$
\forall\, t \in \R, \quad x(t)^{1-\alpha} \leq x(0)^{1-\alpha} + \int_0^t f(s) \, \dd s.
$$
\end{lemma}
Considering \eqref{ready_for_the_petit_lemme}, we are going to apply Lemma \ref{petit_mais_joli} with $\alpha = 1-1/s_0$ and $x(t) = \Norm{y(t)_{> N}}{s_0}^2 + \eps^{\iota(s_0,r)}$ where $\iota(s_0,r)=(3r + \frac{5}{2}-\frac{2r + 1}{s_0})(1-\frac1{s_0})^{-1}$. Indeed, $x$ naturally satisfies the estimate
$$
\frac{\dd}{\dd t} x(t) \leq C   (1+t)^{\frac1{s_0}} \eps^{\frac{2r + 1}{s_0}}  x(t)^{1-\frac1{s_0}}.
$$
Now applying Lemma \eqref{petit_mais_joli}, we get
$$
x(t) \leq C x(0) + C \eps^{2r + 1} (1+t)^{s_0+1} .
$$
Since we have assumed that $\Norm{y(0)_{>N}}{s_0}\leq \eps^{2r + 1}$ and since a straightforward estimate proves that $\iota(s_0,r)\geq 2r+1$, we deduce that 
$$\Norm{y(t)_{>N}}{s_0}\leq C\eps^{2r + 1}(1+t)^{s_0+1}$$
and hence 
 for $t\leq \min(T,\eps^{-\frac r{(s_0+1)}})$ and $\eps$ small enough
\begin{equation}\label{OhYes2}\|y(t)_{>N}\|_{s_0}\leq \frac12 \eps^{r + 1}.\end{equation}
Hence combining \eqref{OhYes1} and \eqref{OhYes2}, we conclude by continuity argument that $T\geq \eps^{-\frac r{(s_0+1)}}$ which finishes the proof. 
\endproof

\end{document}